\documentclass[12pt]{article}
\usepackage{amsmath,amsthm,amssymb}
\usepackage[mathscr]{euscript}
\usepackage{verbatim}
\usepackage{dsfont,bm}
\usepackage{stmaryrd} 
\usepackage{color}
\usepackage{mathtools}
\usepackage{graphicx}
\usepackage{enumerate}
\usepackage{esint} 
\usepackage{mathabx}
\allowdisplaybreaks[4]

\usepackage{xcolor}
\definecolor{darkgreen}{rgb}{0,0.75,0}
\definecolor{darkred}{rgb}{0.75,0,0}
\definecolor{darkmagenta}{rgb}{0.5,0,0.5}
\usepackage[colorlinks,citecolor=darkgreen,linkcolor=darkred,urlcolor=darkmagenta]{hyperref}
\newtheorem{theorem}{Theorem}[section]
\newtheorem{cor}[theorem]{Corollary}
\newtheorem{lem}[theorem]{Lemma}
\newtheorem{prop}[theorem]{Proposition}

\theoremstyle{definition}
\newtheorem{definition}[theorem]{Definition}
\newtheorem{question}[theorem]{Question}

\newtheorem{remark}[theorem]{Remark}
\newtheorem{example}[theorem]{Example}

\newtheorem{conjecture}[theorem]{Conjecture}

\numberwithin{equation}{section}

\def\be{\begin{equation}}
	\def\ee{\end{equation}}
\def\bes{\begin{equation*}}
	\def\ees{\end{equation*}}

\newcommand{\abs}[1]{{\left\lvert #1\right\rvert}}
\newcommand\norm[1]{\left\lVert #1\right\rVert} 
\newcommand{\one}{\mathds{1}} 
\newcommand{\loc}[0]{\operatorname{loc}}

\DeclareMathOperator*{\esssup}{ess\,sup}

\newcommand{\Capa}{\operatorname{Cap}}

 \def\sE {{\mathcal E}} \def\sF {{\mathcal F}}
  
  \def\sL {{\mathcal L}}
\def\sM {{\mathcal M}}

  \def\sX {{\mathcal X}}

 \def\bE {{\mathbb E}} 
 \def\bH {{\mathbb H}} 
  
 \def\bN {{\mathbb N}} 
\def\bP {{\mathbb P}}  \def\bR {{\mathbb R}}
\def\bS {{\mathbb S}}  
  \def\bX {{\mathbb X}}
\def\bY {{\mathbb Y}} \def\bZ {{\mathbb Z}}

\def\ignore#1{}

\def\to {\rightarrow}

\def\dint{\int\kern-.6em\int}

\newcommand\restr[2]{{
		\left.\kern-\nulldelimiterspace 
		#1 
		\vphantom{\big|} 
		\right|_{#2} 
}} 
\newcommand{\Borel}{\mathscr{B}}
\newcommand{\oneptcpt}[1]{#1_{\partial}}

\def\diam{{\mathop{{\rm diam }}}}

\def\supp{\mathop{{\rm supp}}\nolimits}

\newcommand{\on}[1]{\operatorname{ #1}}

\def\wt{\widetilde}

\def\be{\begin{equation}}
	\def\ee{\end{equation}}
\def\bes{\begin{equation*}}
	\def\ees{\end{equation*}}
\def\ba{\begin{align}}
	\def\ea{\end{align}}
\def\xxea{\end{align}}
\def\bas{\begin{align*}}
\def\eas{\end{align*}}

\definecolor{dgreen}{rgb}{0, 0.6, 0.1}
\definecolor{dblue}{rgb}{0, 0.0, 0.6}
\definecolor{vdblue}{rgb}{0,.08, 0.45}
\definecolor{dred}{rgb}{0.7, 0.0, 0.0}
\definecolor{vdblue}{rgb}{0,.08, 0.45}

\definecolor{purple}{rgb}{0.6, 0.0, 0.6}
\definecolor{mytext}{rgb}{0.1, 0.1, 0.1}

\newcommand{\Lip}{\operatorname{Lip}} 
\newcommand{\Ind}{\operatorname{Ind}} 

\begin{document}
	
	\font\titlefont=cmbx14 scaled\magstep1
	\title{\titlefont Martingale and analytic dimensions coincide under Gaussian heat kernel bounds} 
	\author{Mathav Murugan} 
	\renewcommand{\thefootnote}{}
	\footnotetext{The author is partially supported by NSERC and the Canada research chairs program.}
	\renewcommand{\thefootnote}{\arabic{footnote}}
	\setcounter{footnote}{0}
 
	\maketitle
	\vspace{-0.5cm}
\begin{abstract}
Given a strongly local Dirichlet form on a metric measure space that satisfies Gaussian heat kernel bounds, we show that the martingale dimension of the associated diffusion process coincides with      Cheeger's  analytic dimension of the underlying metric measure space. More precisely, we show that the pointwise version of the martingale dimension introduced by Hino (called the pointwise index) almost everywhere equals the pointwise dimension of the measurable differentiable structure constructed by Cheeger.
Using known properties of spaces that admit a measurable differentiable structure,  we show that the martingale dimension is bounded from above by  the Hausdorff dimension of the underlying metric space, thereby extending an earlier bound obtained by Hino for some self-similar sets. 
\end{abstract}
\section{Introduction}

The notions of martingale   and analytic dimensions are motivated by   classical results concerning stochastic integral representation of martingale additive functionals and  differentiability of Lipschitz functions on Euclidean space, respectively. We begin with an informal description of these dimensions, deferring the technical definitions to later sections.

The concept of martingale dimension is motivated by a theorem of Ventcel' that establishes a stochastic integral representation for a class of martingale additive functionals associated with Brownian motion in $\mathbb{R}^n$. Ventcel' demonstrated that every square-integrable martingale additive functional for Brownian motion in $\mathbb{R}^n$
can be expressed as a sum of $n$ stochastic integrals with respect to a chosen set of  $n$ martingales, each corresponding to a component of the Brownian motion in $\mathbb{R}^n$ \cite{Ven}. 
 
More generally, the martingale dimension of a Markov process is the minimal number $k$ such that every martingale additive functional satisfying suitable integrability conditions can be expressed as a sum of $k$ stochastic integrals with respect to a chosen ``basis" of martingale additive functionals (see  \textsection \ref{ss:mdim} and Definition \ref{d:mdim}). This concept was originally developed by Motoo and Watanabe \cite{MW} for martingale additive functionals corresponding to a Markov process and later generalized to a more abstract setting of filtered probability spaces by Davis and Varaiya \cite{DV} under the name \emph{multiplicity of a filtration}.  

 The notion of analytic dimension arises from Cheeger's far-reaching generalization of Radamacher's theorem \cite[Theorem 4.38]{Che}.      Radmacher's theorem asserts that every Lipschitz function on $\bR^n$ is differentiable almost everywhere  with respect to the Lebesgue measure \cite{Rad},\cite[Theorem 3.2]{EG}. 
 In order to describe Cheeger's notion of differentiability, let us recall that if $f:\bR^N \to \bR$ is differentiable at $x \in \bR^N$ if and only if there exists a unique linear functional $L:\bR^N \to \bR$ such that 
 \[
 f(y)= f(x)+ L(y-x) + o(\norm{x-y}), \quad \mbox{as $y \to x$,}
 \]
 where $\norm{\cdot}$ is Euclidean norm.

  Cheeger formulated a notion of differentiability of a Lipschitz function $f:X \to \bR$ on a metric space $(X,d)$ in terms of charts $(U,\phi)$, where $\phi:X \to \bR^N$ is a Lipschitz function and $U$ is a measurable subset of $X$. On a metric space $(X,d)$, differentiability of $f:X \to \bR$ at a point $x \in U$ with respect to $\phi$ is defined as the existence of a unique linear functional $L:\bR^N \to \bR$  such that $f(\cdot)$ coincides with $f(x) + L(\phi(\cdot)-\phi(x))$ up to first-order; that is,
  \[
  f(y)= f(x)+ L (\phi(y)-\phi(x)) +o(d(x,y)) , \quad \mbox{as $y \to x$.}
  \]
   A chart $(U,\phi)$ on a metric measure space $(X,d,m)$  satisfies the property that \emph{every} Lipschitz function $f:X \to \bR$ is differentiable $m$-almost everywhere  on $U$. 
   The existence of a countable family of charts $\{(U_\alpha,\phi_\alpha)\}$ 
   such that $\phi_\alpha:U_\alpha \to \bR^{N(\alpha)}$ with $m(U_\alpha)>0$ and $\cup_{\alpha} U_\alpha=X$ can be viewed as the analogue of almost everywhere differentiability of Lipschitz functions in the setting of metric measure spaces and hence is a version of Radamacher's theorem.
   Such a collection of charts covering the space is referred to as a \emph{measurable differentiable structure}.
    The smallest $N_0$ such that every chart $(U_\alpha,\phi_\alpha)$ as above has dimension $N(\alpha) \le N_0$   is called the analytic dimension of a metric measure space $(X,d,m)$ (see \textsection \ref{ss:mds} and Definition \ref{d:adim}). We refer to \cite{KM16} for a primer on Cheeger's Radamacher theorem.
  
  Spaces that admit a measurable differentiable structure are also known as Lipschitz differentiability spaces \cite{Bat,Dav}.
  There are alternate approaches to differentiability of Lipschitz functions due to D.~Bate and N.~Weaver \cite{Bat,Wea}.
   In \cite{Bat}, Bate developed a different approach to   differentiability of Lipschitz functions by decomposing the measure into $1$-rectifiable measures which are called \emph{Alberti representations}. 
 	Weaver's approach was to define measurable vector fields (called \emph{derivations}) which are operators acting on Lipschitz functions \cite{Wea}. 
 	We refer to the work of Schioppa \cite{Sch} for a comparison between these approaches.
 	
The goal of this work is to prove the following result (see Theorem \ref{t:main} for the precise statement).
		\begin{quote}
		\emph{If a metric measure space admits a symmetric diffusion process that satisfies Gaussian heat kernel bounds, then the martingale dimension of the diffusion process coincides with analytic dimension of the underlying metric measure space.}
	\end{quote}
Gaussian estimates of the heat kernel  for symmetric diffusions have been widely studied and known to hold in a broad range of settings including manifolds with non-negative Ricci curvature \cite{LY}, uniformly elliptic operators \cite{Aro,Sal92}, weighted manifolds \cite[Theorem 7.1]{GrS}, Lie groups of polynomial growth \cite[Chapter IV]{VSC} and many other examples \cite[Section 3.3]{Sal10}. To illustrate the breadth of examples, we note that for any $d_H \in [1,\infty)$ and $d_M \in \bN$ such that $d_H \ge d_M$, there is a   symmetric diffusion process that satisfies Gaussian heat kernel bounds such that the   martingale dimension of the diffusion process is $d_M$ and the Hausdorff dimension of the underlying space is $d_H$ (see Example \ref{x:ltype}). The condition $d_M \le d_H$ is necessary as we show in Corollary \ref{c:dhdm}.

 	Next, we outline some key ideas to establish the equality between martingale and analytic dimensions. 
	Our approach to compute martingale dimension is to analyze energy measures associated to the Dirichlet form. This approach of analyzing energy measures to compute martingale dimension is originally due to Kusuoka \cite{Kus89,Kus} for diffusions on fractals to answer a question of Barlow and Perkins \cite[Problem 10.6]{BP}. This approach was further developed in the general setting of strongly local Dirichlet forms by Hino \cite{Hin08,Hin10,Hin13,Hin13b,Hin14}. More precisely, Hino defines the \emph{index} of a strongly local Dirichlet form using energy measures and shows that it coincides with the martingale dimension (see Definition \ref{d:index} and \cite[Theorem 3.4]{Hin10}).
	Therefore, in order to obtain the equality between martingale and analytic dimensions, it suffices to show the equality between index and analytic dimension.

	This requires us to relate energy measures to the analysis of Lipschitz functions. A key ingredient of our proof are estimates relating energy measures and pointwise Lipschitz constants established by Koskela and Zhou \cite[Theorem 2.2]{KZ12}. However, in order to apply the results of \cite{KZ12} we need a bi-Lipschitz equivalence between the given metric and the intrinsic metric of the Dirichlet form (see the proof of Lemma \ref{l:lipem}). We establish this bi-Lipschitz equivalence by using recent results in \cite{KM20,Mur20} (see the proof of Proposition \ref{p:hk2cons}(ii)). Our analysis leads to a refined equality between pointwise versions of martingale dimension (called the pointwise index) and pointwise version of analytic dimension which is the pointwise dimension of the charts in the measurable differentiable structure.
	
	Our main result establishing the equality between martingale and analytic dimensions has several useful consequences. One new consequence of our work is  that the martingale dimension is finite under Gaussian heat kernel estimates. 
	Moreover, it allows us to extend an inequality between martingale and spectral dimensions established by Hino in \cite{Hin13} without the use of self-similarity (see Corollary \ref{c:finitedim} and Remark \ref{r:main}(b)). Obtaining this estimate was a primary impetus behind our work. In the survey \cite{Hin14}, the author asks to clarify the relationship between the work of Cheeger on measurable differentiable structures \cite{Che} and index of Dirichlet form defined in \cite{Hin10}. Our main result (Theorem \ref{t:main}) clarifies the precise relationship between these works.
	
	The remainder of the work is organized as follows.  In \textsection \ref{ss:df}, we recall basic notions in the theory of Dirichlet forms such as energy measures, intrinsic metric and Gaussian estimates on the heat kernel. In \textsection \ref{ss:mdim} and \textsection\ref{ss:index}, we recall the definitions of martingale dimension and index. In \textsection \ref{ss:mds}, we recall some fundamental notions in Cheeger's work on measurable differentiable structures such as charts, differentiability with respect to charts, analytic dimension and its pointwise version. 
	In \textsection \ref{ss:doubling}, we recall notions of doubling property in a metric measure space. In \textsection \ref{ss:pi}, we recall the definition of Poincar\'e inequalities used in the setting of Dirichlet forms and metric measure spaces. We establish some useful consequences of Gaussian heat kernel bounds in  Proposition \ref{p:hk2cons}.
	We state and prove the main result in \textsection \ref{s:results} and recall some examples in \textsection \ref{ss:examples}. 
	Many interesting open questions concerning martingale dimension remain. We record some of these questions in  \textsection \ref{ss:questions}. 
	
\section{Setting and Preliminaries}

\subsection{Metric measure Dirichlet space and energy measure} \label{ss:df}
	Throughout this paper, we consider a metric space $(X,d)$ in which
$B(x,r):=B_{d}(x,r):=\{y\in X \mid d(x,y)<r\}$ is relatively compact
(i.e., has compact closure) for any $(x,r)\in X\times(0,\infty)$,
and a Radon measure $m$ on $X$ with full support, i.e., a Borel measure
$m$ on $X$ which is finite on any compact subset of $X$ and strictly positive on any
non-empty open subset of $X$. We always assume that $X$ contains at least two elements,
and such a triple $(X,d,m)$ is referred to as a \emph{metric measure space}.

Let $(\mathcal{E},\mathcal{F})$ be a \emph{symmetric Dirichlet form} on $L^{2}(X,m)$;
that is, $\mathcal{F}$ is a dense linear subspace of $L^{2}(X,m)$, and
$\mathcal{E}:\mathcal{F}\times\mathcal{F}\to\mathbb{R}$
is a non-negative definite symmetric bilinear form which is \emph{closed}
($\mathcal{F}$ is a Hilbert space under the inner product $\mathcal{E}_{1}:= \mathcal{E}+ \langle \cdot,\cdot \rangle_{L^{2}(X,m)}$)
and \emph{Markovian} ($f^{+}\wedge 1\in\mathcal{F}$ and $\mathcal{E}(f^{+}\wedge 1,f^{+}\wedge 1)\leq \mathcal{E}(f,f)$ for any $f\in\mathcal{F}$).
Recall that $(\mathcal{E},\mathcal{F})$ is called \emph{regular} if
$\mathcal{F}\cap C_{\mathrm{c}}(X)$ is dense both in $(\mathcal{F},\mathcal{E}_{1})$
and in $(C_{\mathrm{c}}(X),\|\cdot\|_{\mathrm{sup}})$, and that
$(\mathcal{E},\mathcal{F})$ is called \emph{strongly local} if $\mathcal{E}(f,g)=0$
for any $f,g\in\mathcal{F}$ with $\supp_{m}[f]$, $\supp_{m}[g]$ compact and
$\supp_{m}[f-a\one_{X}]\cap\supp_{m}[g]=\emptyset$ for some $a\in\mathbb{R}$. Here
$C_{\mathrm{c}}(X)$ denotes the space of $\mathbb{R}$-valued continuous functions on $X$ with compact support, and
for a Borel measurable function $f:X\to[-\infty,\infty]$ or an
$m$-equivalence class $f$ of such functions, $\supp_{m}[f]$ denotes the support of the measure $|f|\,dm$,
i.e., the smallest closed subset $F$ of $X$ with $\int_{X\setminus F}|f|\,dm=0$,
which exists since $X$ has a countable open base for its topology; note that
$\supp_{m}[f]$ coincides with the closure of $X\setminus f^{-1}(0)$ in $X$ if $f$ is continuous.
The pair $(X,d,m,\mathcal{E},\mathcal{F})$ of a metric measure space $(X,d,m)$ and a strongly local,
regular symmetric Dirichlet form $(\mathcal{E},\mathcal{F})$ on $L^{2}(X,m)$ is termed
a \emph{metric measure Dirichlet space}, or an \emph{MMD space} in abbreviation. By Fukushima's theorem about regular Dirichlet forms, the MMD space corresponds to a symmetric Markov processes on $X$ with continuous sample paths \cite[Theorem 7.2.1 and 7.2.2]{FOT}.
We refer to \cite{FOT,CF} for details of the theory of symmetric Dirichlet forms.

We recall the definition of energy measure.
Note that $fg\in\mathcal{F}$
for any $f,g\in\mathcal{F}\cap L^{\infty}(X,m)$ by \cite[Theorem 1.4.2-(ii)]{FOT}
and that $\{(-n)\vee(f\wedge n)\}_{n=1}^{\infty}\subset\mathcal{F}$ and
$\lim_{n\to\infty}(-n)\vee(f\wedge n)=f$ in norm in $(\mathcal{F},\mathcal{E}_{1})$
by \cite[Theorem 1.4.2-(iii)]{FOT}.

\begin{definition}\label{d:EnergyMeas}
	Let $(X,d,m,\mathcal{E},\mathcal{F})$ be an MMD space.
	The \emph{energy measure} $\Gamma(f,f)$ of $f\in\mathcal{F}$
	associated with $(X,d,m,\mathcal{E},\mathcal{F})$ is defined,
	first for $f\in\mathcal{F}\cap L^{\infty}(X,m)$ as the unique ($[0,\infty]$-valued)
	Borel measure on $X$ such that
	\begin{equation}\label{e:EnergyMeas}
		\int_{X} g \, d\Gamma(f,f)= \mathcal{E}(f,fg)-\frac{1}{2}\mathcal{E}(f^{2},g) \qquad \textrm{for all $g \in \mathcal{F}\cap C_{\mathrm{c}}(X)$,}
	\end{equation}
	and then by
	$\Gamma(f,f)(A):=\lim_{n\to\infty}\Gamma\bigl((-n)\vee(f\wedge n),(-n)\vee(f\wedge n)\bigr)(A)$
	for each Borel subset $A$ of $X$ for general $f\in\mathcal{F}$. The signed measure $\Gamma(f,g)$ for $f,g \in \mathcal{F}$ is defined by polarization.
\end{definition}

\begin{definition}[\hypertarget{hke}{$\on{HKE}(2)$}]\label{d:HKE}
	Let $(X,d,m,\mathcal{E},\mathcal{F})$ be an MMD space, and let $\{P_{t}\}_{t>0}$
	denote its associated Markov semigroup. A family $\{p_{t}\}_{t>0}$ of
	$[0,\infty]$-valued Borel measurable functions on $X \times X$ is called the
	\emph{heat kernel} of $(X,d,m,\mathcal{E},\mathcal{F})$, if $p_{t}$ is the integral kernel
	of the operator $P_t$ for any $t>0$, that is, for any $t > 0$ and for any $f \in L^{2}(X,m)$,
	\begin{equation*}
		P_{t} f(x) = \int_{X} p_{t}(x,y) f (y)\, dm (y) \qquad \textrm{for $m$-a.e.\ $x \in X$.}
	\end{equation*}
	We say that $(X,d,m,\mathcal{E},\mathcal{F})$ satisfies the \textbf{Gaussian heat kernel estimates}
	\hyperlink{hke}{$\on{HKE}(2)$}, if its heat kernel $\{p_{t}\}_{t>0}$ exists and
	there exist $C_{1},c_{1},C_2,c_{2}\in(0,\infty)$ such that for each $t>0$,
	\begin{equation*} \tag*{\hyperlink{hke}{$\on{HKE}(2)$}}
\frac{c_{2}}{m\bigl(B(x,\sqrt{t})\bigr)} \exp \left(  - C_{2}\frac{d(x,y)^2}{ t}   \right)	\le	p_{t}(x,y) \leq \frac{C_{1}}{m\bigl(B(x,\sqrt{t})\bigr)} \exp \left( -  c_{1}\frac{d(x,y)^2}{ t}   \right)
	\end{equation*}
	for $m$-a.e.\ $x,y \in X$.
\end{definition}

\begin{definition}\label{d:dint}
	Let $(X,d,m,\mathcal{E},\mathcal{F})$ be an MMD space. We define its
	\textbf{intrinsic metric} $d_{\on{int}}:X\times X\to[0,\infty]$ by
	\begin{equation}\label{e:dint}
		d_{\on{int}}(x,y) := \sup \bigl\{f(x) -f(y) \bigm|
		\textrm{$f \in \mathcal{F}_{\on{loc}} \cap \mathcal{C}(X)$, $\Gamma(f,f) \leq m$} \bigr\},
	\end{equation}
	where
	\begin{equation}\label{e:Floc}
		\mathcal{F}_{\loc} := \Biggl\{ f \Biggm|
		\begin{minipage}{285pt}
			$f$ is an $m$-equivalence class of $\mathbb{R}$-valued Borel measurable functions
			on $X$ such that $f \one_{V} = f^{\#} \one_{V}$ $m$-a.e.\ for some $f^{\#}\in\mathcal{F}$
			for each relatively compact open subset $V$ of $X$
		\end{minipage}
		\Biggr\}
	\end{equation}
	and the energy measure $\Gamma(f,f)$ of $f\in\mathcal{F}_{\loc}$ associated with
	$(X,d,m,\mathcal{E},\mathcal{F})$ is defined as the unique Borel measure on $X$
	such that $\Gamma(f,f)(A)=\Gamma(f^{\#},f^{\#})(A)$ for any relatively compact
	Borel subset $A$ of $X$ and any $V,f^{\#}$ as in \eqref{e:Floc} with $A\subset V$;
	note that $\Gamma(f^{\#},f^{\#})(A)$ is independent of a particular choice of such $V,f^{\#}$
	by   \cite[Corollary 3.2.1]{FOT} and \cite[(2.1)]{Hin10}.
\end{definition}

\subsection{Martingale dimension} \label{ss:mdim}
Throughout this subsection, we fix an MMD space $(X,d,m,\sE,\sF)$.
We define the \emph{$1$-capacity}
$\Capa_1(A)$ of $A \subset X$ with respect to Dirichlet form $(\sE,\sF)$ on $L^2(X,m)$ by
\begin{equation} \label{e:defCap1}
	\Capa_1(A) := \inf \bigl\{ \sE_{1}(f,f) \bigm\vert \textrm{$f \in \sF$, $f \geq 1$ $m$-a.e.\ on a neighborhood of $A$} \bigr\},
\end{equation}
where $\sE_{1}:=\sE+\langle\cdot,\cdot\rangle_{L^{2}(X,m)}$ as defined before.
A set $N \subset X$ with $\Capa_1(N)=0$ is called an \emph{exceptional set}. 
Let $A \subset X$. A statement depending on $x \in A$ is said to hold \emph{quasi-everywhere} (abbreviated as `q.e.') on $A$ if
there exists a set $N \subset X$ with $\Capa_1(N)$ such that the statement is true for every
$x \in A \setminus N$. 

By a the fundamental theorem of M.\ Fukushima \cite[Theorem 7.2.1]{FOT},
the assumption of the regularity of the Dirichlet form $(\sE,\sF)$ on $L^2(X,m)$ allows us to
associate to the MMD space $(X,d,m,\sE,\sF)$ an $m$-symmetric Hunt process on $X$.
	We recall that a   \emph{Hunt process}  $\sX = (\Omega, \mathscr{M}, \{\sX_{t}\}_{t\in[0,\infty]},\{\bP_{x}\}_{x \in \oneptcpt{X}})$
on $X$ is a right-continuous strong Markov process on $(\oneptcpt{X},\Borel(\oneptcpt{X}))$
which has the left limit $\sX_{t-}(\omega):=\lim_{s\uparrow t}\sX_{s}(\omega)$ in $\oneptcpt{X}$
for any $(t,\omega)\in(0,\infty)\times\Omega$ and is quasi-left-continuous on $(0,\infty)$
(see \cite[Definition A.1.23-(ii) and Theorem A.1.24]{CF}),
where $\oneptcpt{X} = X \cup \{ \partial \}$ denotes
the one-point compactification of $X$. We always consider each function
$f\colon X \to [-\infty,\infty]$ as being defined also at $\partial$
by setting $f(\partial):=0$. Let $\mathscr{F}_{*}=\{\mathscr{F}_{t}\}_{t \in [0,\infty]}$
denote the minimum augmented admissible filtration of $\sX$ in $\Omega$ as defined in \cite[p.\ 397]{CF},
so that $\mathscr{F}_{*}$ is right-continuous, i.e.,
$\mathscr{F}_{t}=\bigcap_{s \in (t,\infty)}\mathscr{F}_{s}$ for any $t \in [0,\infty)$ by \cite[Theorem A.1.18]{CF}.
 For any measure $\mu$ on $X$, by $\bP_\mu$, we denote the measure $\bP_\mu(\cdot):= \int_{X} \bP_x \, d\mu(x)$ and the integrals with respect to $\bP_\mu$ as $\bE_\mu$. 

\begin{definition}[Additive functional]
A collection $A=\{A_{t}\}_{t\in[0,\infty)}$ of $(-\infty,\infty]$-valued random variables
on $\Omega$ is called a \emph{additive functional} (\emph{AF} for short)
of $\sX$, if the following three conditions hold:
\begin{enumerate}[\rm(i)]\setlength{\itemsep}{0pt}\vspace{-5pt}
	\item\label{it:PCAF-adapted} $A_{t}$ is $\mathscr{F}_{t}$-measurable for any $t \in [0,\infty)$.
	\item\label{it:PCAF-qe-as} There exist $\Lambda \in \mathscr{F}_{\infty}$ and a properly exceptional set $\mathcal{N} \subset X$ for $\sX$
	such that $\bP_{x}(\Lambda)=1$ for any $x \in X \setminus \mathcal{N}$ and $\theta_{t}(\Lambda) \subset \Lambda$ for any $t \in [0,\infty)$.
	\item\label{it:PCAF-sample-path-properties} For any $\omega\in\Lambda$, $[0,\infty) \ni t \mapsto A_{t}(\omega)$ is a $[-\infty,\infty]$-valued  function
	with $A_{0}(\omega)=0$ such that for any $s,t \in [0,\infty)$,
	$A_{t}(\omega)<\infty$ if $t < \zeta(\omega)$,
	$A_{t}(\omega)=A_{\zeta(\omega)}(\omega)$ if $t \geq \zeta(\omega)$,
	and $A_{t+s}(\omega)=A_{t}(\omega)+A_{s}(\theta_{t}(\omega))$.
\end{enumerate}
The sets $\Lambda$ and $\mathcal{N}$ are referred to as a \emph{defining set} and
an \emph{exceptional set}, respectively, of the additive functional $A$. 

We say that an additive functional is a positive continuous additive functional (PCAF) if it is non-negative and continuous on its defining set.

We say that an additive functional is a \emph{martingale additive functional} if $\{M_t\}_{t \in [0,\infty)}$ is an additive functional such that for any $t>0$ and q.e.~$x \in X$, we have $\bE_x[M_t^2] <\infty$ and $\bE_x[M_t]=0$.
The \emph{energy of a martingale additive functional} $\{M_t\}_{t \in [0,\infty)}$ is defined as
\begin{equation} \label{e:energyAF}
	e(M):= \lim_{t \downarrow 0} \frac{1}{2t} \bE_m \left(M_t^2\right)
\end{equation}
whenever the limit exists.
Since $t \mapsto \bE_m[M_t^2]$ is subadditive in $t$, the energy $e(M)$ is well-defined and equals $ \sup_{t>0}\frac{1}{2t}\bE_m \left(M_t^2\right)$. 
By $\overset{\circ}{\mathcal{M}}$ we denote the \emph{martingale additive functionals of finite energy}.
\end{definition}
By the strong locality of the Dirichlet form and \cite[Lemma 5.5.1(ii)]{FOT}, each $M \in \overset{\circ}{\mathcal{M}}$ is a continuous additive functional. 
Therefore by \cite[p.~412, Theorem A.3.3]{FOT}, each $M \in \overset{\circ}{\mathcal{M}}$ admits a positive continuous additive functional $\langle M \rangle$ referred to as the \emph{quadratic variation associated with $M$} that satisfies  
\[
\bE_x [\langle M \rangle_t]= \bE_x[M_t^2], \quad \mbox{for all $t>0$, and q.e.~$x \in X$.}
\]
By the Revuz correspondence \cite[Theorems 5.1.3 and 5.1.4]{FOT}, there exists a unique measure $\mu_{\langle M \rangle}$ such that 
for any non-negative Borel functions $f,h:X \to [0,\infty)$, we have 
\[
\bE_{h \cdot m} \left[ \int_0^t f(\sX_s)\,dA_s \right] = \int_0^t \int_X \bE_x \left[h(\sX_s)\right] f(x)\,\mu_{\langle M \rangle}(dx)\,ds.
\]
The measure $\mu_{\langle M \rangle}$ defined above is called the \emph{energy measure} of the martingale additive function $\{M_t\}_{t \in [0,\infty)}$. The energy and energy measure of a martingale additive functional are related by \cite[(5.2.8)]{FOT}
\[
e(M) = \frac{1}{2} \mu_{\langle M \rangle} (X), \quad \mbox{for all $M \in \overset{\circ}{\sM}$}.
\]
By polarization, for $M,L \in \overset{\circ}{\mathcal{M}}$, we define 
\[
e(M,L)= \frac{1}{2} \left(e(M+L)- e(M)-e(L)\right), \quad \mu_{\langle M, L \rangle}=  \frac{1}{2} \left(\mu_{\langle M+L \rangle} - \mu_{\langle M \rangle} - \mu_{\langle L \rangle}\right).
\]
The space of martingale additive functionals of finite energy  $\overset{\circ}{\mathcal{M}}$ equipped with the inner product $e(\cdot,\cdot)$ is a Hilbert space \cite[Theorem 5.2.1]{FOT}.
We recall the definition of stochastic integral with respect to a martingale additive functional \cite[Theorem 5.6.1]{FOT}.
For $M \in  \overset{\circ}{\mathcal{M}}$ and $f \in L^2(X,\mu_{\langle M \rangle})$, we define the \emph{stochastic integral} $f \bullet M$ as the unique element in  $\overset{\circ}{\mathcal{M}}$ such that 
\[
e(f \bullet M, L) = \frac{1}{2} \int_X f(x) \,\mu_{\langle M,L \rangle}(dx), \quad \mbox{for all  $L \in \overset{\circ}{\mathcal{M}}$.}
\]
The above definition of stochastic integral is essentially due to Kunita and Watanabe \cite[Theorem 2.1]{KW}.
\begin{definition} \label{d:mdim} (Martingale dimension, \cite[Definition 3.3]{Hin10})
The \emph{martingale dimension} of $\{\sX_t\}$ (or equivalently of the Dirichlet form $(\sE,\sF)$ on $L^2(X,m)$) is defined as the smallest non-negative integer $p$ that satisfies the following property: there exists a collection $\{M^{(k)}\}_{1 \le k \le p}$ in $\overset{\circ}{\mathcal{M}}$ such that every $M \in \overset{\circ}{\mathcal{M}}$ has a stochastic integral representation as 
\[
M_t = \sum_{k=1}^p \left(h_k \bullet M^{(k)}\right)_t, \quad \mbox{for all $t >0$ and $\bP_x$-almost surely for q.e.~$x \in X$,}
\]
where $h_k \in L^2(X,\mu_{\langle M^{(k)} \rangle})$ for all $k=1,\ldots,p$.
If such a $p$ does not exist, then the martingale dimension is defined as $+\infty$.
\end{definition}
\subsection{Index of a Dirichlet form} \label{ss:index}
In \cite{Kus89, Kus}, Kusuoka demonstrated that certain linear independence properties of energy measures of functions in the domain of the Dirichlet form can be used to compute the martingale dimension on a class of fractals which led to the notion of index for Dirichlet forms corresponding to self-similar sets.
 This idea of Kusuoka was further developed by Hino \cite{Hin10} for a general MMD space  and led to the notion of \emph{index} of a Dirichlet form.
To recall Hino's definition of index, we first define the notion of a minimal energy-dominant measure.
\begin{definition}[{\cite[Definition 2.1]{Hin10}}]\label{d:minimal-energy-dominant}
	Let $(X,d,m,\mathcal{E},\mathcal{F})$ be an MMD space and let $\Gamma(\cdot,\cdot)$ denote the corresponding energy measure. A $\sigma$-finite Borel measure
	$\nu$ on $X$ is called a \textbf{minimal energy-dominant measure}
	of $(\mathcal{E},\mathcal{F})$ if the following two conditions are satisfied:
	\begin{enumerate} 
		\item[(i)](Domination) For every $f \in \mathcal{F}$, we have $\Gamma(f,f) \ll \nu$.
		\item[(ii)](Minimality) If another $\sigma$-finite Borel measure $\nu'$
		on $X$ satisfies condition (i) with $\nu$ replaced
		by $\nu'$, then $\nu \ll \nu'$.
	\end{enumerate}
	Note that by \cite[Lemmas 2.2, 2.3 and 2.4]{Hin10}, a minimal energy-dominant measure of
	$(\mathcal{E},\mathcal{F})$ always exists and is precisely a $\sigma$-finite Borel measure
	$\nu$ on $X$ such that for each Borel subset $A$ of $X$, $\nu(A)=0$ if and only if
	$\Gamma(f,f)(A)=0$ for all $f\in\mathcal{F}$.
\end{definition}

We recall the definition of index associated to a Dirichlet form. 
\begin{definition}\cite[Definition 2.9]{Hin10} \label{d:index}
	Let $(X,d,m,\mathcal{E},\mathcal{F})$ be an MMD space. Let $\Gamma(\cdot,\cdot)$ denote the corresponding energy measure and let $\nu$ be a minimal energy dominant measure.
	\begin{enumerate}[(i)]
		\item The \emph{ pointwise index} is a measurable function $p_H:X \to \bN \cup \{0,\infty\}$ such that the following hold:
		\begin{enumerate}[(a)]
			\item For any $N \in \bN$, $f_1,\ldots,f_N \in \sF$, we have 
			\[
			\on{rank} \left(\frac{d\Gamma(f_i,f_j)}{d\nu}(x)\right)_{1 \le i,j \le N} \le p_H(x) \quad \mbox{for $\nu$-almost every $x \in X$.}
			\]
			\item For any other function $p_H':X \to \bN \cup \{0,\infty\}$ that satisfies (a) with $p_H'$ instead of $p_H$, then $p_H(x) \le p_H'(x)$ for $\nu$-almost every $x \in X$.
		\end{enumerate}
		\item The \textbf{index} of the MMD space $(X,d,m,\sE,\sF)$ is defined as $\nu$-$\esssup_{x \in X} p_H(x)$, where $p_H$ is a pointwise index.
	\end{enumerate}
	It is easy to see that the pointwise index is well-defined in the $\nu$-almost everywhere sense and does not depend on the choice of $\nu$. Therefore the index is well-defined and takes values in $\bN \cup \{0,\infty\}$.
\end{definition}
Hino shows that the index of an MMD space $(X,d,m,\sE,\sF)$ coincides with the martingale dimension of the associated diffusion process \cite[Theorem 3.4]{Hin10}. In \cite[Theorem 3.4]{Hin13b}, Hino interprets the pointwise index as the pointwise dimension of a measurable tangent space.

\subsection{Cheeger's measurable differentiable structure} \label{ss:mds}

\begin{definition}
	Let $(X,d)$ be a metric space. We say that a function $f:X \to \bR$ is \textbf{Lipschitz} if there exists $K \in (0,\infty)$ such that $\abs{f(x)-f(y)} \le K d(x,y)$ for all $x,y \in X$. 
	The vector space of all Lipschitz functions is denoted by $\Lip(X)$.
	The \emph{pointwise Lipschitz constant} of a function $f: X \to \bR$ is given by
	\[\Lip f(x):= \limsup_{r \to 0} \sup_{d(x,y)<r} \frac{\abs{f(x)-f(y)}}{r}= \limsup_{\substack{y \to x,\\
			y \neq x}} \frac{\abs{f(x)-f(y)}}{d(x,y)}.\]
	The following definition provides a notion of differentiability on a metric space. 
\end{definition}

Cheeger discovered a far-reaching generalization of Rachamacher's theorem. 
In order to describe this version of Radamacher's theorem in a metric measure space, we need a suitable notion of differentiability on a metric space that we recall in the definition below.
\begin{definition} \label{d:derivative} Suppose $f: X \to \bR$ and $\phi=(\phi_1,\ldots,\phi_N): X \to\bR^N$ are Lipschitz functions on a metric measure space $(X,d,m)$. Then $f$ is \emph{differentiable with respect to $\phi$ at $x_0 \in X$} if there is a unique $a=(a_1,\ldots,a_N) \in \bR^N$ such that $f$ and the linear combination $a\cdot \phi = \sum_{i=1}^N a_i \phi_i$ agree to first order near $x_0$: 
	\[
	\limsup_{x \to x_0} \frac{\abs{f(x)-f(x_0)-a \cdot(\phi(x)-\phi(x_0))}}{d(x,x_0)}=0.
	\]
	Equivalently,  $\Lip g(x_0)=0$, where $g(\cdot)= f(\cdot)-\sum_{i=1}^N a_i \phi_i(\cdot)$. The tuple $a \in \bR^N$ is the \emph{derivative of $f$ with respect to $\phi$} and will be denoted by $\partial_\phi f(x_0)$.
\end{definition}
Now that we have a notion of differentiability on a metric space, the analogue of \emph{almost everywhere differentiability} for a metric measure space is given by the notion of a \emph{measurable differentiable structure} defined below. As shown in \cite{Che}, this concept below leads to a new proof of the classical Radamacher's theorem.
\begin{definition} \label{d:chart}
	A \emph{chart of dimension $N$} on a metric measure space $(X,d,m)$ is a pair $(U,\phi)$ where:
	\begin{enumerate}[(i)]
		\item $U \subset X$ is measurable, $m(U) >0$, and $\phi:X \to \bR^N$ is Lipschitz.
		\item Every Lipschitz function $f: X \to \bR$ is differentiable with respect to $\phi$ at $m$-almost every $x_0 \in U$ and the derivative defines a measurable function $\partial_\phi f: U \to \bR^N$.
	\end{enumerate}
A \textbf{measurable differentiable structure}	on $(X,d,m)$ is a  countable collection $\{(U_\alpha,\phi_\alpha)\}$ of charts with uniformly bounded dimension such that $X = \cup_{\alpha} U_\alpha$.
\end{definition}
We note that the notion of charts and measurable differentiable structure are invariant under a bi-Lipschitz change of metric.
\begin{remark} \label{r:bilip}
	Let $\{(U_\alpha,\phi_\alpha)\}$ be a collection of charts that defines a measurable differentiable structure on a metric measure space  $(X,d,m)$. Let $\wt{d}: X \times X \to [0,\infty)$ be a metric on $X$ that is \emph{bi-Lipschitz equivalent} to $d$; that is, there exists $L \in [1,\infty)$ such that $L^{-1}d(x,y) \le \wt{d}(x,y) \le L d(x,y)$ for all $x,y \in X$. Then $\{(U_\alpha,\phi_\alpha)\}$ is a collection of charts that defines a measurable differentiable structure on $(X,\wt{d},m)$.
	 Similarly, it is evident that $(U,\phi)$ is a chart on $(X,d,m)$ if and only if $(U,\phi)$ is a chart on $(X,\wt{d},m)$.
	 	Moreover, for any $f \in \Lip(X)$, the differentiability of $f$ at a point $x \in X$ and the value of the derivative of $f$ with respect to $\phi$ are both invariant under bi-Lipschitz change of metric for any chart $(U,\phi)$.  
\end{remark}

Given a metric measure space, a measurable differentiable structure need not exist \cite[Proposition B.1]{KM16}. However, Cheeger showed that a large class of metric measure spaces admits a measurable differentiable structure \cite[Theorem 4.38]{Che} (see also the exposition of \cite[Theorem 1.3]{KM16}). We refer to \cite[\textsection 1.3]{KM16} for a discussion of different class of examples arising from Euclidean spaces, Carnot groups, glued spaces, Laakso spaces, Bourdon-Pajot buildings, Ricci limit spaces and spaces satisfying synthetic lower bounds on Ricci curvature.

Even if a measurable differentiable structure exists, the associated collection of charts is far from being unique. 
Nevertheless, Cheeger showed that the dimension of the chart of a measurable differentiable structure is well-defined at almost every point. 
Let  $\{(U_\alpha,\phi_\alpha): \alpha \in I \}$ be a measurable differentiable structure on $(X,d,m)$, where $I$ is a countable index set. Then there is a measurable function $d_C: X \to \bN$ such that 
$d_C(x)= N(\alpha)$ for all $\alpha \in I$, $m$-almost every $x \in U_\alpha$, where $N(\alpha)$ is the dimension of the chart $(U_\alpha,\phi_\alpha)$. Furthermore,   up to sets of $m$-measure zero this function $d_C$ does not depend on the choice of the measurable differentiable structure \cite[p. 458]{Che}. This function $d_C$ can be interpreted as the (almost everywhere defined)  pointwise dimension of the fibers of the $L_\infty$ cotangent bundle constructed by Cheeger \cite[p.~458, Definition 4.42]{Che}. 
\begin{definition} \label{d:adim}
	Let $(X,d,m)$ be a metric measure space that admits a measurable differentiable structure.
	 We call the $m$-almost everywhere well-defined function $d_C: X \to \bN$   above as the \textbf{pointwise dimension of the measurable differentiable structure} on $(X,d,m)$. 
	 The \textbf{analytic dimension} of a metric measure space $(X,d,m)$ that admits a measurable differentiable structure is defined as $m$-$\esssup_{x \in X} d_C(x)$.
\end{definition}
From the discussion above,  the analytic dimension of a metric measure space that admits a measurable differentiable structure is well-defined. 

\subsection{Metric and volume doubling properties} \label{ss:doubling}
The   volume doubling property plays an important role in obtaining Gaussian heat kernel estimates as well in Cheeger's generalization of Radamacher's theorem. We recall its definition and the closely related metric doubling property.
\begin{definition}[\hypertarget{vd}{$\on{VD}$}]
	Let $(X,d,m)$ be a metric measure space. We say that $(X,d,m)$ satisfies the
	\textbf{volume doubling property} \hyperlink{vd}{$\on{VD}$},
	if there exists a constant $C_{D}>1$ such that for all $x \in X$ and all $r>0$,
	\begin{equation} \tag*{\hyperlink{vd}{$\on{VD}$}}
		0<m(B(x,2r)) \leq C_{D} m(B(x,r)) < \infty.
	\end{equation}
	Note that if $(X,d,m)$ satisfies \hyperlink{vd}{$\on{VD}$}, then $B(x,r)$ is relatively
	compact (i.e., has compact closure) in $X$ for all $(x,r)\in X\times(0,\infty)$
	by virtue of the completeness of $(X,d)$.
\end{definition}
If $(X,d,m)$ satisfies the volume doubling property, 
then $(X,d)$ satisfies the \textbf{metric doubling property}; that is, there exists $N \in \bN$ such that every ball of radius $r$ can be covered by $N$ balls of radii $r/2$. It is well-known that any metric space satisfying the metric doubling property has finite Hausdorff dimension. 

\subsection{Poincar\'e inequalities} \label{ss:pi}
Similar to the volume doubling property, Poincar\'e inequalities play an important role both in obtaining heat kernel estimates and also in establishing the existence of a measurable differentiable structure. 
The formulation used for these two purposes are different although related as we will see.
\begin{definition} \label{d:pi2}
	We say that $(X,d,m,\mathcal{E},\mathcal{F})$ satisfies the \textbf{Poincar\'e inequality} \hypertarget{pi}{$\operatorname{PI}(2)$},
if there exist constants $C,K\ge 1$ such that 
for all $(x,r)\in X\times(0,\infty)$ and all $f \in \mathcal{F}$,
\begin{equation} \tag*{$\operatorname{PI}(2)$}
	\int_{B(x,r)} (f -   f_{B(x,r)})^2 \,dm  \le C r^2\int_{B(x,K r)}d\Gamma(f,f),
\end{equation}
where $f_{B(x,r)}:= m(B(x,r))^{-1} \int_{B(x,r)} f\, dm$ and $\Gamma$ denotes the energy measure.
\end{definition}
The following Poincar\'e inequality concerns Lipschitz functions on the metric measure space $(X,d,m)$.
\begin{definition} \label{d:ppi}
		We say that $(X,d,m)$ is said to support a $(1,p)$-\emph{Poincar\'e inequality} with $p \in [1,\infty)$ if there exists constants $K \ge 1, C >0$ such that for all $u \in \on{Lip}(X), x \in X$ and $r>0$,
\[
\fint_{B(x,r)} \abs{u-u_{B(x,r)}} \, dm \le C r \left[\fint_{B(x,Kr)}\left( \on{Lip}(u) \right)^p\,dm \right]^{1/p},
\]
where $\fint_A f\,dm$ denotes $\frac{1}{m(A)} \int_A f\,dm$ and $u_{B(x,r)}= \fint_{B(x,r)} u\,dm$.
\end{definition}

We record some useful consequences of Gaussian heat kernel estimates.
The well-known necessary condition in Proposition \ref{p:hk2cons}(i)  is due to Saloff-Coste \cite{Sal} for Brownian motion on Riemannian manifolds and was later extended to general Dirichlet form in \cite{Stu}. In the setting of MMD spaces, it is often assumed   that the metric $d$ is the intrinsic metric (e.g., \cite{Stu,KZ12}) but it turns out that  such an assumption is not necessary as we clarify the relationship between the given metric and intrinsic metric in  Proposition \ref{p:hk2cons}(ii) below.  All these consequences of Gaussian heat kernel estimates are known to experts and follow easily from known results and arguments.

\begin{prop}  \label{p:hk2cons}
Let $(X,d,m,\sE,\sF)$ be an MMD space that satisfies Gaussian heat kernel estimates \hyperlink{hke}{$\on{HKE}(2)$}. 
\begin{enumerate}[(i)]
	\item  (\cite{Sal,Stu})
Then $(X,d,m)$ satisfies the volume doubling property  \hyperlink{vd}{$\on{VD}$} and the MMD space $(X,d,m,\sE,\sF)$ satisfies the Poincar\'e inequality \hyperlink{pi}{$\operatorname{PI}(2)$}. 
\item  (\cite{Mur20,KM20}) The metric $d$ is bi-Lipschitz equivalent to the intrinsic metric $d_{\rm{int}}$ of the MMD space $(X,d,m,\sE,\sF)$.
\item (\cite{KZ12}) The space $\Lip(X) \cap C_c(X)$ satisfies  $\Lip(X) \cap C_c(X) \subset \sF$ and   is dense in the Hilbert space $(\sF,\sE_1)$. Furthermore, $\Lip(X) \subset \sF_{\on{loc}}$.
\item (\cite{KZ12}) The metric measure space $(X,d,m)$ supports a $(1,2)$-Poincar\'e inequality (in the sense of Definition \ref{d:ppi}).
\end{enumerate}
\end{prop}
\begin{proof}
\begin{enumerate}[(i)]
	\item The volume doubling property follows from the argument in \cite{Sal} by integrating the Gaussian lower bound over suitable ball (see also \cite[p.~161]{Sal02}). Saloff-Coste obtains Poincar\'e inequality in \cite[p.~33]{Sal} by adapting an argument of Kusuoka and Stroock \cite{KS17}. The same argument also applies in our setting \cite{Stu}.
	
	\item We note that $(X,d,m)$ satisfies a reverse volume doubling property due to \cite[Corollary 2.3]{Mur20}. By   \cite[Corollary 1.8]{Mur20},  the metric $d$ satisfies the chain condition (cf. \cite[Definition 1.1]{Mur20}). By  \cite[Theorem 1.2]{GHL}, we obtain the necessary cut-off energy inequality in order to apply \cite[Proposition 4.8]{KM20} and conclude that $d$ and $d_{\rm{int}}$ are bi-Lipschitz equivalent. 
	\item The density of  $\Lip(X) \cap C_c(X)$ in $\sF$ follows from (i), (ii), and \cite[Theorem 2.2 (i)]{KZ12}. The result $\Lip(X) \subset \sF_{\on{loc}}$ follows from \cite[Theorem 2.1]{KZ12}, (i) and (ii). We note that (ii) is used because the results of \cite{KZ12} only apply to the intrinsic metric.
	\item This follows from (i), (ii), \cite[Proposition 2.1]{KZ12}, and invariance of Poincare inequalities in Definitions \ref{d:pi2} and \ref{d:ppi} under a bi-Lipschitz change of metric.
	\qedhere
\end{enumerate}
We state Cheeger's generalization of Radamacher's theorem. We will use this to show that any MMD space that has Gaussian heat kernel bounds admits a measurable differentiable structure.  
\begin{theorem} (\cite[Theorem 4.38]{Che}) \label{t:cheeger}
	Let $(X,d,m)$ be a complete metric space that satisfies the volume doubling property \hyperlink{vd}{$\on{VD}$} and supports a $(1,p)$-Poincar\'e inequality for some $p \in [1,\infty)$. Then there is a measurable differentiable structure on  $(X,d,m)$ and the analytic dimension is bounded by a constant that depends only on the constants involved in the assumptions.
\end{theorem}
\end{proof}
The $(1,p)$-Poincar\'e inequality in Cheeger's work \cite[(4.3)]{Che} is different from that in Definition \ref{d:ppi}. Nevertheless, these $(1,p)$-Poincar\'e inequalities are known to be equivalent due to a result of Keith \cite[Theorem 2]{Kei03} (see also  \cite[Theorem 8.4.2]{HKST}). We refer to \cite{Kei} for a  different proof of Cheeger's theorem and the survey \cite{KM16} for a nice exposition of the key ideas involved in the proof of Theorem \ref{t:cheeger}.

\section{Comparing martingale and analytic dimensions} \label{s:results}

It is known that if an MMD space satisfies Gaussian heat kernel bounds, then the symmetric measure is a minimal energy dominant measure and the underlying metric measure space admits a measurable differentiable structure as we recall next.
 Proposition \ref{p:main}(i) is a special case of a more general result in \cite{KM20} while (ii) follows easily from results in \cite{Che,KZ12,KM20}. 
\begin{prop} \label{p:main}
	Let $(X,d,m,\sE,\sF)$ be an  MMD space that satisfies Gaussian heat kernel estimates 	\hyperlink{hke}{$\on{HKE}(2)$}. Then the following hold:
	\begin{enumerate}[(i)]
		\item  (cf. \cite[Propositions 4.5 and 4.7]{KM20}) $m$ is a minimal energy dominant measure. In particular,  the pointwise index $p_H(\cdot)$ of  $(X,d,m,\sE,\sF)$ is well-defined $m$-almost everywhere.
		\item (cf. \cite[Proposition 2.1]{KZ12} and \cite[Theorem 4.38]{Che}) $(X,d,m)$ admits a measurable differentiable structure.
	\end{enumerate}
\end{prop}
\begin{proof}
	\begin{enumerate}[(i)]
		\item The first claim is a special case of \cite[Propositions 4.5 and 4.7]{KM20} while the second claim follows from the definition of pointwise index.
		\item This is an immediate consequence of Proposition \ref{p:hk2cons}(i),(iv) and Cheeger's Radamacher theorem recalled in Theorem \ref{t:cheeger}.
	\end{enumerate}
\end{proof}
The above proposition implies that under Gaussian heat kernel bounds, the pointwise index of the MMD space $p_H(\cdot)$ and the pointwise dimension of the measurable differentiable structure of the underlying metric measure space are both well-defined almost everywhere with respect to the reference measure. 
The following is the main result of this work and shows the almost everywhere equality of  pointwise dimension of the measurable differentiable structure constructed by Cheeger and pointwise index in the sense of Hino.
\begin{theorem}\label{t:main}
	Let $(X,d,m,\sE,\sF)$ be an  MMD space that satisfies Gaussian heat kernel estimates 	\hyperlink{hke}{$\on{HKE}(2)$}. Then the   The pointwise index $p_H(\cdot)$ of the MMD space $(X,d,m,\sE,\sF)$  agrees  with the pointwise dimension $d_C(\cdot)$ of a measurable differentiable structure on  $(X,d,m)$; that is 
	\begin{equation} \label{e:eqptwise}
		p_H(x)=d_C(x) \quad \mbox{for $m$-almost every $x \in X$.}
	\end{equation}
	In particular, the    martingale dimension of the associated diffusion process  coincides with the analytic dimension of $(X,d,m)$.
\end{theorem}
The proof of Theorem \ref{t:main} needs further preparation. Before we prove it, we state a consequence of our main result.
Using known results on analytic dimension, we obtain the following result that implies that the martingale dimension is bounded from above by the Hausdorff dimension under Gaussian heat kernel bounds. Even the finiteness of martingale dimension stated below is new to the best of the author's knowledge. 
\begin{cor} \label{c:finitedim}
		Let $(X,d,m,\sE,\sF)$ be a  MMD space that satisfies Gaussian heat kernel estimates 	\hyperlink{hke}{$\on{HKE}(2)$}. Then the martingale dimension of the associated diffusion process is bounded from above by the Hausdorff dimension of the metric space $(X,d)$.
		In particular, the martingale dimension is finite.
\end{cor}
\begin{proof}
	This is an immediate consequence of Theorem \ref{t:main} and known bounds on analytic dimension in terms of Hausdorff dimension follows from  \cite[Theorem 5.3]{BKO} and \cite[Theorem 6.6]{Bat}.  
\end{proof}

\begin{remark} \label{r:main}
\begin{enumerate}[(a)]
	\item The bound in Corollary \ref{c:finitedim} is sharp and is the \emph{only constraint} between Hausdorff and martingale dimensions. Furthermore, for any $d_M \in \bN, d_H \in [d_M,\infty)$, there exists an MMD space $(X,d,m,\sE,\sF)$ that satisfies Gaussian heat kernel estimates such that the martingale dimension of the associated diffusion is $d_M$ and the Hausdorff dimension of the underlying metric space $(X,d)$ is $d_H$ (see Example \ref{x:ltype}).
	
	\item If the symmetric measure  is \emph{$Q$-Ahlfors regular} (that is, there exists $C \in (1,\infty)$ such that  $C^{-1} r^Q \le m(B(x,r)) \le C r^Q$ for all $x \in X, r < \diam(X,d)$), then $Q$ is the Hausdorff dimension of $(X,d)$ and also under Gaussian heat kernel bounds $Q$ is also the \emph{spectral dimension} (cf. \cite[Definition 3.23]{Bar} for the terminology and \cite[(3.50)]{Bar} for the justification of this terminology). In this case, Corollary \ref{c:finitedim} implies that martingale dimension is less than or equal to spectral dimension.  Such an inequality between spectral and martingale dimensions was first obtained by Hino \cite[Theorem 3.5]{Hin13} for diffusions on some self-similar sets. Unlike \cite{Hin13} we do not assume self-similarity, however our space-time scaling for the heat kernel is Gaussian. We conjecture that a similar estimate is true if we replace Gaussian estimates with the more general sub-Gaussian heat kernel estimate (see Conjecture \ref{c:dsdm}).
	\item A slightly weaker bound on analytic dimension by the Assouad dimension of the underlying metric space was shown earlier in \cite[Corollary 4.6]{Sch}( (see also \cite[ Corollary 8.5]{Dav} for a slightly weaker bound in terms of Assouad dimension).  There is yet another upper bound on analytic dimension in terms of Lipschitz dimension introduced by Cheeger and Kleiner \cite{CK} due to David \cite[Theorem 7.6]{Dav21}. In general, the Hausdorff and Lipschitz dimensions are not comparable \cite[{Propositions 3.8 and Theorem 6.8}]{Dav21}.
\end{enumerate}
\end{remark}

We recall the notion of linear dependence in an infinitesimal sense which plays an important role in the construction of a measurable differentiable structure.
\begin{definition}
	An $N$-tuple of functions $\mathbf{f}=(f_1,\ldots,f_N)$, where $f_i : X \to \bR$ for $ 1 \le i \le N$ is \emph{infinitesimally dependent} at $x \in X$ if there exists $\lambda \in \bR^{N} \setminus \{0\}$ such that 
	$$\Lip(\lambda\cdot \mathbf{f})(x)=0.$$ 
	We denote the set where $\mathbf{f}: X \to \bR^N$ is \emph{not infinitesimally dependent} by $\Ind(\mathbf{f})$.
\end{definition}
	We note that $\Ind(f)$ is a Borel measurable set by \cite[Lemma 7.2.3]{Kei}.

The following is a slight strengthening of property (a)  of pointwise index in Definition \ref{d:index}.
\begin{lem} \label{l:locrank}
	Let $(X,d,m,\sE,\sF)$ be an MMD space and let $\nu$ be a minimal energy dominant measure.  For any $N \in \bN$, $f_1,\ldots,f_N \in \sF_{\on{loc}}$, we have 
	\[
	\on{rank} \left(\frac{d\Gamma(f_i,f_j)}{d\nu}(x)\right)_{1 \le i,j \le N} \le p_H(x) \quad \mbox{for $\nu$-almost every $x \in X$.}
	\]
\end{lem}
\begin{proof}

Fix $x_0 \in X$. By the assumption that balls are precompact, for any  $n \in \bN$,
there exists $g_1,\ldots,g_N \in \sF$ such that $g_i=f_i$ $m$-almost everywhere on $B(x_0,n)$.
Therefore by property (a) of Definition \ref{d:index} and the strong locality of $(\sE,\sF)$ (in particular, \cite[Corollary 3.2.1]{FOT} and \cite[(2.1)]{Hin10}),   we have 
	\[
\on{rank} \left(\frac{d\Gamma(f_i,f_j)}{d\nu}(x)\right)_{1 \le i,j \le N} = \on{rank} \left(\frac{d\Gamma(g_i,g_j)}{d\nu}(x)\right)_{1 \le i,j \le N} \le p_H(x)  
\]
for $\nu$-almost every $x \in B(x_0,n)$. This leads to the desired conclusion as $X = \cup_{n \in \bN} B(x_0,n)$.
\end{proof}

The following lemma is at the heart of our proof and is essentially due to Koskela and Zhou \cite{KZ12}.
It is the key estimate that serves as a bridge between the analysis of the   Dirichlet form with the analysis of Lipschitz functions on the underlying metric measure space. Ultimately, this leads to the relation between martingale and analytic dimensions.
It is helpful to recall that $\Lip(X) \subset \sF_{\on{loc}}$ from Proposition \ref{p:hk2cons}.
\begin{lem} \label{l:lipem}
	Let $(X,d,m,\sE,\sF)$ be an  MMD space that satisfies Gaussian heat kernel estimates 	\hyperlink{hke}{$\on{HKE}(2)$}. 
	There exists $C \in [1,\infty)$ such that 
	\begin{equation}  \label{e:}
C^{-1} \sqrt{\frac{d\Gamma(f,f)}{dm}(x)} \le \Lip f(x) \le  C \sqrt{\frac{d\Gamma(f,f)}{dm}(x)} 
	\end{equation}
	for all $f \in \Lip(X)$, and for $m$-almost every $x \in X$.
\end{lem}
\begin{proof}
	It follows from \cite[Theorem 2.2(ii)]{KZ12} and the bi-Lipschitz equivalence of $d$ and $d_{\rm{int}}$ in Proposition \ref{p:hk2cons}(ii).
\end{proof}

For an $N$-tuple $f=(f_1,\ldots,f_N) \in (\sF_{\on{loc}})^N$, we define $N\times N$-positive semi-definite matrix valued function $M_f : X \to \bR^{N \times N}$ that is well-defined $m$-almost everywhere as 
\begin{equation}\label{e:defmat}
	  M_f(x)= \left( \frac{d \Gamma(f_i,f_j)}{dm}(x)\right)_{1 \le i,j \le N}.
\end{equation}
The proof of equality between pointwise index and pointwise dimension of Cheeger's measurable differentiable structure can be divided into a matching lower and upper bounds on the pointwise index. We begin with the \emph{lower} bound on the pointwise index in the Proposition below. 
\begin{prop} \label{p:lbindex}
	Let $(X,d,m,\sE,\sF)$ be an  MMD space that satisfies Gaussian heat kernel estimates 	\hyperlink{hke}{$\on{HKE}(2)$}. 
	Let $(U,\phi)$ be a chart of dimension $N \in \bN$ on $(X,d,m)$. Then 
	$p_H(x) \ge N$ for $m$-almost every $x \in U$.
\end{prop}
\begin{proof}
	Let $(\phi_1,\ldots,\phi_N) = \phi$ denote the components of the chart.
Let $M_\phi:X \to \bR^{N \times N}$ be as defined by \eqref{e:defmat}.	By Lemma \ref{l:locrank}, it suffices to show that 
	\begin{equation}\label{e:lb1}
	\on{rank}(M_\phi(x))=N, \quad \mbox{for $m$-almost every $x \in U$.}
	\end{equation}

	Let $\Lambda$ be a countable dense subset of the unit sphere $\bS^{N-1}$ in $\bR^N$. We claim that 
	\begin{equation} \label{e:lb2}
		\inf_{\lambda \in \Lambda} \Lip(\lambda\cdot \phi)(x) >0 \quad \mbox{for $m$-almost every $x \in U$.}
	\end{equation}
	In order to prove \eqref{e:lb2}, consider the function $g \equiv 0$  that is identically zero. Since $g$ is differentiable with respect to $\phi$ at $m$-almost every point of $U$, there exists a measurable set $V$ such that $m(U \setminus V)=0$ such that $g$ is differentiable with respect to $\phi$ at all $x \in V$ with derivative $\partial_\phi g(x)=0$ for all $x \in V$. The \emph{uniqueness} of the derivative implies that 
	\begin{equation} \label{e:lb2a}
	\Lip(\lambda\cdot\phi)(x)=\limsup_{y \to x} \frac{\abs{\lambda\cdot\phi(y)- \lambda\cdot\phi(x)}}{d(y,x)}\neq 0, \quad \mbox{for all $x \in V, \lambda \in \Lambda$.}
	\end{equation}
	By the continuity of $\lambda \mapsto \Lip(\lambda\cdot\phi)(x)$on $\bR^N$  for each $x \in X$ (see \cite[Sublemma 7.2.4]{Kei}) and since $\overline{\Lambda}=\bS^{N-1}$ is compact, we have 
	\[
	\inf_{\lambda \in \Lambda} \Lip(\lambda\cdot \phi)(x)= \min_{\lambda \in \bS^{N-1}}\Lip(\lambda\cdot \phi)(x) \stackrel{\eqref{e:lb2a}}{>}0 \quad \mbox{for all $x \in V$.}
		\]
	This concludes the proof of \eqref{e:lb2}.
	
	By \eqref{e:lb2} and Lemma \ref{l:lipem} (recalling that $\Lambda$ is  countable), we have 
	\begin{equation} \label{e:lb3}
		\inf_{\lambda \in \Lambda} \frac{d \Gamma(\lambda\cdot \phi,\lambda\cdot \phi)}{dm}(x) = \inf_{\lambda \in \Lambda} \lambda \cdot (M_\phi(x)\lambda) >0 \quad \mbox{for $m$-almost every $x \in U$.}
	\end{equation}
	By \eqref{e:lb3} and the fact that $M_\phi(\cdot)$ is symmetric, non-negative definite valued matrix, we conclude \eqref{e:lb1}.
\end{proof}
In the following proposition, we establish an upper bound of the pointwise index matching the lower bound obtained in Proposition \ref{p:lbindex}.
\begin{prop} \label{p:ubindex}
	Let $(X,d,m,\sE,\sF)$ be an  MMD space that satisfies Gaussian heat kernel estimates 	\hyperlink{hke}{$\on{HKE}(2)$}. 
	Let $(U,\phi)$ be a chart of dimension $N \in \bN$ on $(X,d,m)$. Then 
	$p_H(x) \le N$ for $m$-almost every $x \in U$.
\end{prop}
\begin{proof}
	Assume to the contrary that  	$p_H(x) \ge  N+1$ for $m$-almost every $x \in V_0$ where $V_0 \subset U$ and $m(V_0) >0$. By \cite[Lemma 2.5(ii)]{Hin10} and the density of $\Lip(X) \cap C_c(X)$ in the Hilbert space $(\sF,\sE_1)$ (cf. \cite[Remark 4.6]{KM20}), there exists $f_1,\ldots,f_{N+1} \in \Lip(X) \cap C_c(X) \subset \sF$ and $V_1 \subset V_0$ with $m(V_1)>0$ such that $f=(f_1,\ldots,f_{N+1})$ satisfies
	\begin{equation} \label{e:ub1}
\on{rank} (M_f(x)) = N+1 \quad \mbox{for $m$-almost every $x \in V_1$,}
	\end{equation}
	where $M_f$ is as defined in \eqref{e:defmat}. 
	Let $\Lambda \subset \bS^N \subset \bR^{N+1}$ be a countable dense subset of the unit sphere $\bS^N$.
	By \eqref{e:ub1} and the fact that $M_f$ is a positive semi-definite matrix valued function, there exists $\delta>0$ and a measurable set $V_2 \subset V_1$ such that $m(V_2)>0$
	and 
	\begin{equation} \label{e:ub2}
		\inf_{\lambda \in \Lambda} \lambda\cdot \left(M_f(x) \lambda\right) \ge \delta \quad \mbox{for $m$-almost every $x \in V_2$.}
	\end{equation}
	By Lemma \ref{l:lipem}, there exists $c>0$, $V_3 \subset V_2$ such that $m(V_2 \setminus V_3)=0$
	\[
	\frac{d\Gamma(\lambda\cdot f,\lambda\cdot f)}{dm}(x) = 	\lambda \cdot \left(M_f(x)\lambda\right) \quad \mbox{for all $x \in V_3, \lambda \in \Lambda$,}
	\]
	 and
	\begin{equation} \label{e:ub3}
		\Lip(\lambda\cdot f)(x) \ge c \sqrt{\delta} \quad \mbox{for all $x \in V_3, \lambda \in \Lambda$.}
	\end{equation}
	Since for each $x \in X$, the function $\lambda \mapsto \Lip(\lambda\cdot f)(x)$ is continuous on $\bR^{N+1}$ by \cite[Sublemma 7.2.4]{Kei}, we can improve \eqref{e:ub3} to 
	\begin{equation*}
	\inf_{\lambda \in \bS^{N}}	\Lip(\lambda\cdot f)(x) \ge c   \sqrt{\delta}  \quad \mbox{for all $x \in V_3$.}
	\end{equation*}
	Equivalently, this means that there exists $V_3 \subset U$ with $m(V_3)>0$ with
	\begin{equation} \label{e:ub4}
		V_3 \subset \Ind(f). 
	\end{equation}
	By the argument in \cite[p. 311]{Kei} using \cite[Sublemma 7.3.5]{Kei}
	there exists $K \in \bN$ such that is  the largest number with the following property: $K \ge N+1$ and there exists $g=(g_1,\ldots,g_K)$ such that $g_i \in \Lip(X)$ for all $i=1,\ldots,K$ and $g_i=f_i$ for all $i=1,\ldots,N+1$ and
	\begin{equation} \label{e:ub5}
	 	m\left(\Ind(g) \cap V_3\right) >0.
	\end{equation}
 We recall that $\Ind(g)$ is measurable by \cite[Lemma 7.2.3]{Kei}. Setting $W=\Ind(g) \cap V_3$ and following the same argument as \cite[Proofs of Sublemmas 7.3.6, 7.3.7 and 7.3.8]{Kei}, we conclude that $(W,g)$ is a chart.
 Since $(U,\phi)$ is a chart of dimension $N$ and $(W,g)$ is a chart of dimension $K$ with $m(U \cap W)=m(W)>0$, by \cite[p. 458]{Che} we have $N=K \ge N+1$, which implies the desired contradiction.
\end{proof}
We can now complete the proof of our main result.
\begin{proof}[Proof of Theorem \ref{t:main}]
	The equality \eqref{e:eqptwise} follows immediately from the matching lower and upper bounds in Propositions \ref{p:lbindex} and \ref{p:ubindex}. The equality between index of the MMD space and analytic dimensions follow from (ii) and \eqref{e:eqptwise}. By Hino's theorem on the equality between the index of an MMD space and the martingale dimension of the associated diffusion process (cf. \cite[Theorem 3.4]{Hin10}), we obtain the desired conclusion.
\end{proof}
\subsection{Examples} \label{ss:examples}
For Brownian motion on Euclidean space the martingale dimension coincides with the Hausdorff and topological dimensions of the underlying space. In general, all these three dimensions may be different as shown in the example below.
\begin{example}[Horizontal Brownian motion on  Heisenberg group]
	We consider the $3$-dimensional Heisenberg group $\bH=\{(x,y,z): x,y,z \in \bR\}$ equipped with the group operation 
	\[
	(x_1,y_1,z_1) \odot (x_2,y_2,z_2)= (x_1+x_2, y_1+y_2, z_1 + z_2 + x_1y_2-x_2y_1).
	\]
	The Lebesgue measure $m$ on $\bR^3$ is the (left and right) Haar measure. The following left-invariant vector fields forms a basis of the Lie algebra:
	\[
	\bX= \partial_x - y \partial_z, \quad \bY= \partial_y + x \partial_z, \quad \bZ= \partial_z.
	\]
	The sub-elliptic Laplacian
	\[
	L = \frac{1}{2}(\bX^2+\bY^2)= \frac{1}{2} \left(\partial^2_{xx}+ \partial^2_{yy}+2 x\partial^2_{yz}- 2 y \partial^2_{xz} +  (x^2+y^2) \partial^2_{zz}\right)
	\]
	is the generator of a diffusion process that is closely related to the  \emph{L\'evy area} $S(t)$ of the two-dimensional Brownian motion $(B_1(t),B_2(t))$, where 
	\[
	S(t)= \int_0^t\left( B_1(s)\,dB_2(s) - B_2(s)\,dB_1(s) \right).
	\]
	Then $(B_1(t),B_2(t),S(t))$ is the Markov process generated by $L$ and can be viewed as a Brownian motion on $\bH$ (see \cite[\textsection 2.1.2]{BDW}). To be precise, the operator $\frac{1}{2}(\bX^2+\bY^2)= \frac{1}{2} \left(\partial^2_{xx}+ \partial^2_{yy}+2 x\partial^2_{xy}- 2 y \partial^2_{xy} +  (x^2+y^2) \partial^2_{zz}\right)$ on $C_c^\infty(\bH)$ is an essentially self-adjoint operator as outlined in \cite[p. 950]{DGS} and $L$ is defined to the unique   self-adjoint extension. By \cite[ Theorem 1.3.1]{FOT}, the generator $L$ defines a Dirichlet form $(\sE,\sF)$ on $L^2(\bH,m)$. It is regular as   $C_c^\infty(\bH)$ is a core.
	Let $d$ denote the intrinsic metric which also turns out to be the Carnot-Carath\'eodory metric. The MMD space $(\bH,d,m,\sE,\sF)$ satisfies Gaussian heat kernel bounds \cite[Th\'eor\`eme 1.1]{Li}.
	
	It is known that the $(\bH,d)$ is homeomorphic to the $\bR^3$ with respect to the Euclidean metric and hence the \emph{topological dimension of $(\bH,d)$ is 3}. 
	However the measure $m$ is $4$-Ahlfors regular on $(\bH,d)$ and hence the \emph{Hausdorff dimension of $(\bH,d)$ is 4}. The \emph{analytic dimension} (and hence the \emph{martingale dimension}) is 2. This follows from the fact that a martingale additive function with respect to the associated diffusion can be viewed as that associated with the 2-dimensional Brownian motion. Alternately, this follows from a version of Radamacher's theorem due to Pansu \cite{Pan}. As pointed out in \cite{KM16}, $(\bH,d,m)$ carries a measurable differentiable structure with a single chart $(\bH,\phi)$, where $\phi:\bH \to \bR^2$ is defined by $\phi(x,y,z)=(x,y)$.	
\end{example}

The next example  illustrates that every possible joint values of martingale and Hausdorff dimensions satisfying the inequality in the conclusion of Corollary \ref{c:finitedim} is possible.
\begin{example}[Laakso-type spaces and product with Euclidean spaces] \label{x:ltype}
	Let $d_m \in \bN, d_H \in  [1,\infty)$ be such that $d_m \le d_H$. We outline a construction of an MMD space with martingale dimension $d_m$ and Hausdorff dimension $d_H$.
	If $d_m=1$, then this follows from \cite[Theorem 5.4, Lemma 5.6]{Mur24+} as there exists a Laakso-type metric measure space $(\sL,d_{\sL},m_{\sL})$
	admitting a Dirichlet form that satisfies Gaussian heat kernel bounds and such that $m_{\sL}$ is $d_{H}$-Ahlfors regular and hence the Hausdorff dimension is $d_H$. By \cite[ Proposition 5.17]{Mur24+}, the martingale dimension is $1$.
	If $d_m \ge 2$, we consider the product space $\sL \times \bR^{d_m-1}$, where $\sL$ is a $d_H-(d_m-1)$-Ahlfors regular Laakso-type space that satisfies Gaussian heat kernel estimate as explained above. The Euclidean space $\bR^{d_m-1}$ is equipped with the Euclidean metric and Lebesgue measure. Then we consider the diffusion on the product space such that the projections to the components are independent diffusions on the Laakso-type space and the standard Brownian motion on the Euclidean space. The metric and measure on the product space are taken to be product measure. One can verify that the corresponding MMD space satisfies Gaussian heat kernel estimates, has the desired martingale dimension $d_m$ and the Hausdorff dimension $d_H$. 
\end{example}

\subsection{Related questions} \label{ss:questions}
 Many basic questions concerning martingale dimension still remain open. For instance, whether or not the martingale dimension is finite is not known in many situations. We conjecture the following finiteness of martingale dimension referring the reader to \cite[Definition 2.3]{Mur24+} for the definition of sub-Gaussian heat kernel estimates (such estimates allow for a more general space-time scaling and is a generalization of Gaussian bounds).
\begin{conjecture} \label{c:dmfinite}
If an MMD space $(X,d,m,\sE,\sF)$ that satisfies sub-Gaussian heat kernel estimates where $m$ is a doubling measure, then the martingale dimension is finite.
\end{conjecture}
 Our main result (see Remark \ref{r:main}(b)) along with Hino's bound in \cite[Theorem 3.5]{Hin13} provides  evidence towards the following quantitative version of  Conjecture \ref{c:dmfinite}. We refer the reader to \cite{Bar} or \cite{Mur24+} for the definitions of volume growth exponent and walk dimension used in the following conjecture.
\begin{conjecture} \label{c:dsdm}
	If an MMD space $(X,d,m,\sE,\sF)$ satisfies sub-Gaussian heat kernel estimates with volume growth exponent $\alpha$ and walk dimension $\beta$, then the martingale dimension $d_m$ satisfies $d_m \le \frac{2\alpha}{\beta}$.
\end{conjecture}

Since the walk dimension $\beta$ is greater than or equal to two, we have the following weaker version of Conjecture \ref{c:dsdm}. 
\begin{conjecture} \label{c:dhdm}
	If an MMD space $(X,d,m,\sE,\sF)$ satisfies sub-Gaussian heat kernel estimates, then the martingale dimension $d_m$ is less than or equal to the Hausdorff dimension of $(X,d)$.
\end{conjecture}
One outcome of our work relating martingale  and analytic dimensions is that 
questions concerning one of these notions of dimension have a natural analogue in the other. 
Therefore this provides opportunity for interactions between different areas.
We hope that this connection would help with a better understanding of both these notions and development of techniques that uses ideas from these different viewpoints. For concreteness, we list an analogue of a question due to Kleiner and Schioppa concerning analytic dimensions \cite[Question 1.3]{KS17}. The following question is about the existence of exotic diffusions on $\bR^n$ with $n \ge 2$ with martingale dimension one.
\begin{question} \label{q:ks}
	Let $n \ge 2$. Is there an MMD space $(\bR^n,d,m,\sE,\sF)$ that satisfies Gaussian heat kernel estimates with martingale dimension one such that $(\bR^n,d)$ is homeomorphic to $\bR^n$ with the Euclidean metric? (cf. \cite[Question 1.3]{KS17}).
\end{question}
We do not know the answer to this question  even when $n=2$.   On the other hand, the results of Kleiner and Schioppa suggests that there exist  diffusions with martingale dimension one satisfying Gaussian heat kernel estimates  where the underlying space has arbitrarily high topological dimension \cite[Theorem 1.1]{KS17}. We note that Question \ref{q:ks} remains open if we relax the requirement of Gaussian heat kernel bounds to that of sub-Gaussian heat kernel bounds.

\noindent \textbf{Acknowledgments.} 
I thank Aobo Chen for useful discussions related to this work and Yizhou Wang for remarks on an earlier draft. I am grateful to Pietro Wald for several useful comments and references such as \cite{BKO,CK,Dav21}. 
The author is grateful to the anonymous referee for a careful reading and helpful suggestions.
\\

\noindent Department of Mathematics, University of British Columbia,
Vancouver, BC V6T 1Z2, Canada. \\
mathav@math.ubc.ca 

\end{document}